\documentclass[11pt, twoside, leqno]{article}

\usepackage{amsmath,amsthm}
\usepackage{amssymb}

\usepackage[applemac]{inputenc}

\usepackage{enumerate}

\usepackage{graphicx}
\usepackage{hyperref}

\newtheorem{thm}{\indent Theorem}[section]
\newtheorem{cor}[thm]{\indent Corollary}
\newtheorem{lem}[thm]{\indent Lemma}
\newtheorem{prop}[thm]{\indent Proposition}

\theoremstyle{definition}

\numberwithin{equation}{section}

\def\boxit#1#2{\setbox1=\hbox{\kern#1{#2}\kern#1}%
\dimen1=\ht1 \advance\dimen1 by #1 \dimen2=\dp1 \advance\dimen2 by #1
\setbox1=\hbox{\vrule height\dimen1 depth\dimen2\box1\vrule}%
\setbox1=\vbox{\hrule\box1\hrule}%
\advance\dimen1 by .4pt \ht1=\dimen1
\advance\dimen2 by .4pt \dp1=\dimen2 \box1\relax}

\def\hfl#1#2{\smash{\mathop{\hbox to 12 mm{\rightarrowfill}}
\limits^{\scriptstyle#1}_{\scriptstyle#2}}}

 \def\hlfl#1#2{\smash{\mathop{\hbox to 12 mm{\leftarrowfill}}
\limits^{\scriptstyle#1}_{\scriptstyle#2}}}

\def\phfl#1#2{\smash{\mathop{\hbox to 8 mm{\rightarrowfill}}
\limits^{\scriptstyle#1}_{\scriptstyle#2}}}

 \def\phlfl#1#2{\smash{\mathop{\hbox to 8 mm{\leftarrowfill}}
\limits^{\scriptstyle#1}_{\scriptstyle#2}}}

\def\cqfd{\unskip\kern 6pt\penalty 500
\raise -2pt\hbox{\vrule\vbox to 10pt{\hrule width 4pt
\vfill\hrule}\vrule}\par}

\def\house#1{\setbox1=\hbox{$\,#1\,$}%
\dimen1=\ht1 \advance\dimen1 by 2pt \dimen2=\dp1 \advance\dimen2 by 2pt
\setbox1=\hbox{\vrule height\dimen1 depth\dimen2\box1\vrule}%
\setbox1=\vbox{\hrule\box1}%
\advance\dimen1 by .4pt \ht1=\dimen1
\advance\dimen2 by .4pt \dp1=\dimen2 \box1\relax}

\def\virgule{\raise 2pt \hbox{$,$}}

\def\R{\mathbb{R}}
\def\Z{\mathbb{Z}}

\def\og{\leavevmode\raise.3ex\hbox{$\scriptscriptstyle
\langle\!\langle$}}

\def\fg{\leavevmode\raise.3ex\hbox{$\scriptscriptstyle
\,\rangle\!\rangle$}}

\begin{document}

\baselineskip=17pt

\title{{\bf Representation of integers \\ by cyclotomic binary forms} }

\author{{\sc \' Etienne Fouvry}\\
Laboratoire de Math\'{e}matiques d'Orsay \\
 Universit\'{e} Paris--Sud \\
CNRS, Universit\'{e} Paris--Saclay\\
 F--91405 ORSAY, \sc France\\
 E-mail: etienne.fouvry@u-psud.fr\\
 \null
 \\
 {\sc Claude Levesque}\\
D\'{e}partement de math\'{e}matiques et de statistique$\quad$ \\
 Universit\'{e} Laval, Qu\'{e}bec, Qc \\
Canada G1V 0A6 \\
 E-mail: cl@mat.ulaval.ca\\
 \null
 \\
 {\sc Michel Waldschmidt }\\
 Sorbonne Universit\'es \\
 UPMC Univ Paris 06 \\
 UMR 7586 IMJ-PRG \\
 F--75005 Paris, \sc France \\
 E-mail: michel.waldschmidt@imj-prg.fr}

%\date{\misajour}

\maketitle

 \begin{center}

{\large {\sf Dedicated to Robert Tijdeman \\
on the occasion of his $75^{th}$ birthday}}
 \end{center}
 \vskip -1.5cm

\renewcommand{\thefootnote}{}



\renewcommand{\thefootnote}{\arabic{footnote}}
\setcounter{footnote}{0}
%%%%%%%%%%%%%%%%%%%%%%%%%%%%%%%%%%%%%%%%%%%%%%

\begin{abstract}
The homogeneous form $\Phi_n(X,Y)$ of degree $\varphi(n)$ which is associated with the cyclotomic polynomial $\phi_n(X)$ is dubbed a {\it cyclotomic binary form}. 
A positive integer $m\ge 1$ is said to be {\it representable by a cyclotomic binary form} if there exist integers $n,x,y$ with $n\ge 3$ and $\max\{|x|, |y|\}\ge 2$ such that $\Phi_n(x,y)=m$. 
We prove that the number $a_m$ of such representations of $m$ by a cyclotomic binary form is finite. More precisely, we have $\,\varphi(n) \le ( {2}/ {\log 3} )\log m\, $ and $\,
\max\{|x|,|y|\} \le ({2}/{\sqrt{3}} )\, m^{1/\varphi(n)}.\,$ We give a description of the asymptotic cardinality of the set of values taken by the forms for $n\geq 3$. This will imply that the set of integers $m$ such that $a_m\neq 0$ has natural density 0. We will deduce that the average value of the integers $a_m$ among the nonzero values of $a_m$ grows like $\sqrt{\log \, m}$. 
%It turns out that the sequence $(a_m)_{m\geq 1}$ is unbounded because of the role played by the quadratic forms $\Phi_3$, $\Phi_4$ and $\Phi_6$.
\end{abstract}

\section{Introduction} \label{S:Introduction}

 K.~ Gy\H{o}ry obtained in \cite{Gyory} many interesting results on the representation of integers (resp. algebraic integers) by binary forms. He obtained sharp estimates, in contrast with the exponential bounds previously obtained on Thue's equations by means of Baker's results on lower bounds for linear forms in logarithms of algebraic numbers. 
 The bibliography of \cite{Gyory} contains a useful selection of articles dealing with these problems, including \cite{Nagell1} and \cite{Nagell2}. Most particularly, Gy\H{o}ry considered binary forms of degree $d$ with integral coefficients, 
$$
F(X,Y)=a_0X^d+a_1X^{d-1}Y+\cdots+a_{d-1}XY^{d-1}+a_dY^d,
$$ 
which are products of $\ell$ irreducible forms, assuming that the roots of $F(X,1)$ are totally imaginary quadratic numbers over a totally real number field, and he proved that for $m\neq 0$, the solutions $(x,y)\in\Z^2$ of $F(X,Y)=m$ satisfy 
$$
|x|\le 2 |a_d|^{1-(2\ell-1)/d}|m|^{1/d}
\quad\hbox{and}\quad
|y|\le 2 |a_0|^{1-(2\ell-1)/d}|m|^{1/d}.
$$
In other words, the splitting field of each irreducible factor of $F(X,1)$ is a CM-field, {\it i.e.}, a totally imaginary quadratic extension of a totally real number field. In particular, cyclotomic fields are such number fields.

Examples of such binary forms with $a_0=a_d=1$ are given by the cyclotomic binary forms, which we define as follows. 
 
 For $n\ge 1$, denote by $\phi_n(X)$ the cyclotomic polynomial of index $n$ and degree $\varphi(n)$ (Euler's totient function). Following Section 6 of \cite{Nagell2}, the {\it cyclotomic binary form} $\Phi_n(X,Y)$ is defined by $\Phi_n(X,Y)=Y^{\varphi(n)}\phi_n(X/Y)$. In particular, we have $\Phi_n (x,y) >0$ for $n \geq 3$ and $(x,y) \not= (0,0)$ (see \S \ref{cn} below).
 
 In the special case of cyclotomic binary forms, Gy\H{o}ry \cite{Gyory} gives 
$$\max\{|x|,|y|\} \le 2|m|^{1/\varphi(n)}$$
 for the integral solutions $(x,y)$ of $\Phi_n(X,Y) =m$. In contrast with our Theorem $\ref{Thm:1} $ below, Gy\H{o}ry \cite{Gyory} gives an upper bound for $n$ only if $\max\{|x|,|y|\} \ge 3$.

Here is our first main result, in which we exclude the 
cases $n = 1$ and $n = 2$ for which the cyclotomic polynomial $\phi_n$ is linear.

\medskip

\begin{thm}\label{Thm:1}
Let $m$ be a positive integer and let $n,x,y$ be rational integers satisfying $n\ge 3$, $\max\{|x|,|y|\}\ge 2$ and $\Phi_n(x,y)=m$.Then
$$
\varphi(n) \le \frac{2}{\log 3} \log m
\quad\hbox{and}\quad
\max\{|x|,|y|\} \le \frac{2}{\sqrt{3}} \, m^{1/\varphi(n)}.
$$
\end{thm}

In particular, there is no solution when $m\in\{1,2\}$.
\par

From the following lower bound for $\varphi(n)$, proved in six lines in \cite{MW}, namely 
$$
\varphi(n)>\left( \frac{n}{2.685}\right)^{1/1.161},
$$
we deduce that the upper bound $\varphi(n) <2(\log m)/\log 3$ of Theorem $\ref{Thm:1} $ implies
\begin{equation}\label{Equation:1.161}
n<5.383 (\log\, m)^{1.161}.
\end{equation}

Theorem $\ref{Thm:1}$ is a refinement of Gy\H ory's above mentioned result for these cyclotomic binary forms. 
Subject to $\mbox{\sc gcd}(x,y)=1$, Nagell (see Lemma 1, p. 152 of \cite{Nagell1}) comes up with a slightly larger bound than ours for $\varphi(n)$, namely he has
$\varphi(n) <(4\log m)/(3 \log 2)$, and he does not exhibit a bound for 
$ \max\{|x|,|y|\}$.

The estimates of Theorem $\ref{Thm:1}$ are optimal because for $\ell\ge 1$, 
$$
\Phi_3(\ell,-2\ell)=3\ell^2.
$$
If we assume $\varphi(n)>2$, namely $\varphi(n) \geq 4$,
the conclusion of Theorem $\ref{Thm:1}$ can be replaced by 
$$
\varphi(n) \le \frac{4}{\log11} \log m
\quad\hbox{and}\quad
\max\{|x|,|y|\} \le \frac{2}{\root 4 \of {11}} \, m^{1/\varphi(n)}
$$
thanks to
($\ref{Equation:sqrt3maxxy}$).
Again these estimates are best possible since for $\ell \geq 1$, we have 
$\Phi_5(\ell,-2\ell)=11\ell^4$.

There are infinitely many integers $n$ such that $\Phi_n(1,2)<2^{\varphi(n)}$; for instance, $n=2\cdot 3^e$ with $e\ge 1$. 
We will prove the following.

\begin{thm}\label{Thm:1suite}
For $\theta \in \; ]0,1[$, there are only finitely many triples $(n,x,y)$ with $n\geq 3$ and $\max\{|x|,|y|\}\ge 2$, such that $\Phi_n(x,y)\le 2^{\theta \varphi(n)}$; these triples can be effectively determined and they satisfy $\max\{|x|,|y|\} = 2$. 
\end{thm}

As a matter of fact, we shall see that the conclusion $\max\{|x|,|y|\} = 2$ follows from the weaker assumption 
 $$
\Phi_n(x,y)<7^{\varphi(n)/2},
$$
which is optimal since $\Phi_3(1,-3)=7$.
 
Theorem $\ref{Thm:1}$ shows that, for each integer $m\ge 1$, the set 
$$
\big\{(n,x,y)\in\Z^3\; \mid \; n\ge 3,\; \max\{|x|,|y|\}\ge 2, \; \Phi_n(x,y)=m\big\}
$$
is finite. The finiteness of the subset of $(n,x,y)$ subject to the stronger condition $\max\{|x|,|y|\}\ge 3$ follows from \cite{Gyory}, but not for $ \max\{|x|,|y|\}\ge 2$.
%We do cover the case $\max\{ |x|, |y|\} =2$, which is an improvement upon \cite{Gyory}. 
Let us denote by $a_m$ the number of elements in the above set. 
The positive integers $m$ such that $a_m\ge 1$ are the integers which are represented by a cyclotomic binary form. 
%It will be the sequence \cite{OEIS5} of OEIS. 
We will see in \S\ref{S:NumericalComputations} that the sequence of integers $m\ge 1$ such that $a_m\ge 1$ starts with the following values of $a_m$: %\medskip
 
\begin{center}
\noindent
\begin{tabular}
{| c || c | c | c| c| c| c| c|c|c|c|c|c|c|c|c|c|c|c|}
 \hline 
$m$ & $3$ & $4$ & $5$ & $7$ &$8$ & $9$& $10$&$11$&$12$&$13$&$16$&$17$&$18$&$19$&$20$
\\ \hline 
 $a_m$& $8$ & $16$ & $8$ & $24$&$4$ &$16$& $8$&$8$&$12$&$40$&$40$&$16$&$4$&$24$&$8$
 \\ \hline 
\end{tabular}\bigskip

{\bf Table 1}
\end{center} 

\noindent The only result in this direction that we found in the literature is $a_1=0$:
see \cite{Gyory,Nagell1,Nagell2}.

For $N\ge 1$ and $n\geq 3$ let $\mathcal A (\Phi_n;N)$ be the set of positive integers $m\leq N$ which are in a restricted image of $\mathbb Z^2$ by $\Phi_n $. In other words, for $n\geq 3$ we define
\begin{multline*}
\mathcal A (\Phi_n; N) := \bigl\{ m\in \mathbb N\,\mid \, m\leq N,\, m= \Phi_n (x,y) \ \text{ for some } (x,y)\in \mathbb Z^2\\ \text{ with }
\max (\vert x \vert , \vert y \vert) \geq 2
\bigr\}.
\end{multline*}
The following theorem describes the asymptotic cardinality of the set of values taken by the polynomials $\Phi_n$ for $n\geq 3$. Defining
$$
\mathcal A( \Phi_{\{n\geq 3\}}; N) := \bigcup_{n\geq 3} \mathcal A (\Phi_n; N),
$$
we have

\begin{thm}\label{Thm:2}
There exist two sequences $\,(\alpha_h)\,$ and $\,(\beta_h)\,$ (with $\alpha_0 >0$ and $\beta_0 >0$), such that for every 
$M \geq 0$, the following equality holds uniformly for $N \geq 2$:
\begin{multline}\label{infexpan}
\bigl\vert \,\mathcal A( \Phi_{\{n\geq 3\}}; N)\, \bigr\vert =
 \frac{N}{(\log N)^\frac{1}{2}} 
 \left\{ 
\left( \, \alpha_0 - \frac{\beta_0}{(\log N)^\frac{1}{4}}\,\right) + 
 \frac{1}{\log N} \left( \, \alpha_1 - \frac{\beta_1}{(\log N)^\frac14}\, \right) + \cdots \right.\\[4mm] 
+\left. \frac{1}{(\log N)^M}
 \left( \, \alpha_M - \frac{\beta_M}{(\log N)^\frac14}\, \right) +
 O \left( \frac{1}{(\log N)^{M+1}}\right)
 \right\}.
\end{multline} 
\end{thm}
\eject 

The proof of this theorem will be given in \S \ref{proofthm13} with the precise definitions of the coefficients $\alpha_0$ and $\beta_0$. 
This proof will show that the largest contribution to $\bigl\vert \,\mathcal A( \Phi_{\{n\geq 3\}}; N)\, \bigr\vert $ comes from the sets 
$ \mathcal A( \Phi_{ 3}; N) $ and $ \mathcal A( \Phi_{ 4}; N) $.

It follows from Theorem $\ref{Thm:2}$ that the set of integers $m$ such that $a_m\not=0$ has natural density $0$. Combining Theorem $\ref{Thm:2}$ with Lemma \ref{Lemma:varphi>2}, we will deduce that the average value of $a_m$ among the nonzero values of $a_m$ grows like $ \sqrt{\log m}$. More precisely, we have the following.
 
 \begin{cor}\label{Corollary}
For $N\ge 1$, define $A_N$ and $M_N$ by 
$$
A_N = \bigl\vert \,\mathcal A( \Phi_{\{n\geq 3\}}; N)\, \bigr\vert \;\mbox{ and } \; 
M_N=\frac{1}{A_N}(a_1+a_2+\cdots+a_N) .
$$
Then there exists a positive absolute constant $\kappa_1$ such that 
$$
M_N\sim \kappa_1 \sqrt{\log N}\, .
$$
 \end{cor}
 
In particular, the sequence $(a_m)_{m\ge 1}$ is unbounded; this follows from the fact that the number of representations of a positive integer by the quadratic form $\Phi_4(X,Y)$ is an unbounded sequence. The same is true for the quadratic forms $\Phi_3(X,Y)$ and $\Phi_6(X,Y)$. 

In Lemma \ref{Lemma:varphi>2}, we will prove that the number $C_N$ of integers $\le N$ which are represented by a binary form $\Phi_n(X,Y)$ with $\varphi(n)>2$ and $\max\{|x|,|y|\}\ge 2$ is less than
$$
\kappa_2 N^{\frac 12} 
$$
where $\kappa_2$ is a positive absolute constant.

For $m\ge 1$, denote by $b_m$ the number of elements in the set
$$
\big\{(n,x,y)\in\Z^3\; \mid \; \varphi(n)>2 ,\; \max\{|x|,|y|\}\ge 2, \; \Phi_n(x,y)=m\big\}.
$$
 We will see in the last section that for $m$ between 1 and 100, there are exactly 16 values of $m$ for which $b_m$ is different from 0; they are the following ones:\smallskip 

\begin{center}
\begin{small}
\noindent
\begin{tabular}
{| c || c | c | c| c|c|c|c|c|c|c|c|c|c|c|c| c|c|}
 \hline 
$m$ & $11$ & $13$ & $16$ & $17$&$31$&$32$&$43$&$55$& $57$&$61$&$64$&$73$&$80$&
$81$&$82$&$97$\\[2mm]
 \hline 
 $b_m$& $8$ & $8$ & $24$ & $8$ &$8$&$4$&$8$&$8$&$8$&$16$&$24$&$16$&$4$&$24$&$8$&$8$\\[1mm]
 \hline 
\end{tabular}
\end{small}
\end{center}
%\medskip
\centerline{{\bf Table 2}}
\smallskip
 
\begin{lem} \label{Lemme:bmubnbounded}
We have
$$
\limsup_{m\rightarrow\infty}\frac{b_m\log\log\log m}{\log\log m}\ge 8
$$
whereupon the sequence $(b_m)_{m\ge 1}$ is unbounded. 
\end{lem}

\begin{proof}[\indent \sc Proof.]
For the $s$-$th$ odd prime $p_s$, let us consider the integer 
$$k_s=\varphi( 3\cdot 5\cdots p_s),
$$
the product being taken over all the primes between 3 and $p_s$. 
Set $m_s=2^{k_s}$. Then $\Phi_n(x,y)=m_s$ for at least $8s$ values of $(n,x,y)$, namely
$$
(\ell,0,\pm 2^t),\quad (\ell,\pm 2^t,0),\quad (2\ell,0,\pm 2^t),\quad (2\ell,\pm 2^t,0),
$$
for each prime $\ell$ between $3$ and $ p_s$ with $t=k_s/\varphi(\ell)$. Therefore, by
excluding $\ell =3$ we have 
$b_{m_s}\geq 8(s-1)$.

Because
$$
\log k_s =\sum_{3\leq p\leq p_s} \log (p-1),
$$
the Prime Number Theorem implies that for $s\rightarrow \infty$ we have
$$
\log k_s \sim p_s \sim s\log s, 
$$
hence
$$ 
s\sim \frac{\log k_s}{\log \log k_s }\quad \mbox{with} \quad k_s=\frac{\log m_s}{\log 2}
 $$
and 
$$
s\sim \frac{\log \log m_s}{\log \log \log m_s} \cdot
$$ 
This completes the proof of Lemma \ref{Lemme:bmubnbounded}.
 \end{proof}

 %%%%%%%%%%%%%%%%%%%%%%%%%%%%%%%%%%%%%%%%%%%%%%%%%%%%%%%%%
 \section{{\large Positive definite binary forms}}

Consider a Thue equation $F(X,Y)=m$ associated with the polynomial $f(X)$ defined by 
$f(X)=F(X,1)$, where the polynomial $f(X)$ has no real roots and has positive values on $\R$. It happens that this is the case for the cyclotomic polynomials. Such a situation was also considered in \cite{Gyory}. The following result shows that the study of the associated Diophantine equation $F(X,Y)=m$ reduces to finding a lower bound for the values of $f(t)$ on $\R$. \smallskip

 \begin{lem}\label{Lemma:PositiveDefiniteBinaryForms}
Let $f(X)\in\Z[X]$ be a nonzero polynomial of degree $d$ which has no real root. Let $g(X)=X^df(1/X)$.
 Assume that the leading coefficient of $f(X)$ is positive, so that the real numbers, defined by 
$$ \left\{
\begin{array}{llll}\displaystyle
\gamma_1=\inf_{t\in\R} \,f(t), \quad&
\gamma_2=\displaystyle\inf_{t\in\R}\, g(t),\quad\\[3mm]
\gamma'_1=\displaystyle\inf_{-1\le t\le 1}\, f(t),\quad&
\gamma'_2=\displaystyle \inf_{-1\le t\le 1} \,g(t),\quad& 
\gamma'= \min\bigl \{\gamma'_1,\gamma'_2\bigr\} ,
\end{array} \right.
$$
are $>0$. Let $F(X,Y)$ be the binary form $Y^df(X/Y)$ associated with $f(X)$. 
\\
\mbox{$(1)$} Then for each $(x,y)\in\Z^2$, we have 
\begin{equation*} %\label{Eq:upeerbounds}
F(x,y) \ge \gamma_1 |y|^d,
\quad
F(x,y) \ge \gamma_2 |x|^d ,\quad 
F(x,y) \ge \gamma'\max\bigl\{ |x|^d,|y|^d\bigr\}.
\end{equation*}
$(2) $ Moreover, the following statements hold true:

\mbox{\rm (i)} For any real number $c_1$ with $c_1>\gamma_1$, 
there exist an infinite set of couples $(x,y)$ in $\Z\times \Z$ satisfying $y>0$ and 
$$
F(x,y)<c_1 y^d.
$$

 \mbox{\rm (ii)} Further, for any real number $c_2$ with $c_2>\gamma_2$, there exist an infinite set of couples $(x,y)$ in $\Z\times \Z$ satisfying $x>0$ and 
$$
F(x,y)<c_2 x^d.
$$

\mbox{\rm (iii)} Furthermore, for any real number $c$ with $c> \gamma'$, there exist an infinite set of couples $(x,y)$ in $\Z\times \Z$ satisfying 
$$
F(x,y)<c \max\bigl\{|x|^d,|y|^d\bigr\}.
$$
\end{lem}

Before proceeding with the proof, some remarks are in order.
For $|t|>1$, from $g(t)=t^df(1/t)$ we deduce $f(1/t)\le g(t)$. Hence 
$$
\inf_{-1\le t\le 1}|f(t)|\le \inf_{|t|\ge 1}|g(t)|.
$$
Therefore, if we set 
$$ 
\gamma''_1=\inf_{|t|\ge 1} f(t), \quad \gamma''_2=\inf_{|t|\ge 1} g(t),
$$ 
then we have
$$
\gamma_1=\min\{\gamma'_1,\gamma''_1\},
\quad \gamma_2=\min\{\gamma'_2,\gamma''_2\},
\quad \gamma'_2\le \gamma''_1, \quad \gamma'_1\le \gamma''_2.
$$
Hence 
$$
\gamma'=
\min\{\gamma'_1,\gamma'_2\}\le\min\,\{\gamma''_1,\gamma''_2\}\le \max\,\{\gamma_1,\gamma_2\}.
$$
It follows that for a reciprocal polynomial $f$ we have $\gamma_1=\gamma_2=\gamma'_1=\gamma'_2=\gamma'$; in particular, for a reciprocal polynomial, we have 
\begin{equation}\label{equation:reciproque}
\inf_{t\in\R} f(t) = \inf_{|t|\le 1} f(t).
\end{equation}

 \begin{proof}[\indent \sc Proof of Lemma \ref{Lemma:PositiveDefiniteBinaryForms} .]
(1) The proof of the first two lower bounds of the first part is direct. Let us prove the third one. It is plain that
$$
F(x,y) \ge \gamma'_1 |y|^d \quad \hbox{for}\quad 
 |x|\le |y|\quad \hbox{and}\quad 
 F(x,y) \ge \gamma'_2 |x|^d 
 \quad \hbox{for}\quad |x|\ge |y|.
 $$
The third lower bound follows. 

(2) In the second part of the lemma, we claim that the lower bounds of part (1) are optimal. 

 (i) Suppose that $t_0\in\R$ is a value such that $f(t_0)=\gamma_1$.
There exists a real number $a>0$ such that, for $t$ in the open interval $]t_0-a,t_0+a[$, we have 
$$
|f(t)-\gamma|\le (|f'(t_0)|+1)(t-t_0).
$$
For $y>0$, let $x$ in $\Z$ such that 
$$
\left|t_0-\frac{x}{y}\right|\le 1.
$$
For $y$ sufficiently large, $x/y$ is in the interval $]t_0-a,t_0+a[$ and we have 
$$
|F(x,y)-y^df(t_0)|\le (|f'(t_0)|+1) y^{d-1}.
$$
As a consequence, for $y$ sufficiently large, we have
$$ 
\centerline{$ \displaystyle\displaystyle F(x,y)<c_1 y^d. $ }
 $$ 
 
(ii) The next result is proved in the same way. 
 
 (iii) 
Let us prove now the last statement. Assume first
 $c> \gamma'_1$. Let us uppose $-1\le t_0\le 1$. Our argument above gives infinitely many couples $(x,y)$ in $\Z\times\Z$ with 
 $F(x,y)<c |y|^d $ and $|y|\le |x|$. Hence 
$$F(x,y)<c \max\bigl\{|x|^d,|y|^d \bigr \}.$$
 The same argument, starting with $|t_0|\ge 1$, gives infinitely many couples $(x,y)$ with 
 $F(x,y)<c |x|^d $ and $|x|\le |y|$. The case $c> \gamma'_2$ is proved in the same way. Hence the result. 
 \end{proof} 

 Let us mention in passing that Gy\H{o}ry (page 364 of \cite{Gyory}) exhibited Thue equations which 
 have as many (nonzero) solutions as one pleases, by allowing the degree to be large enough. Let us complement with a similar example. 
Let $c_j$ ($j=1,2,\dots,\ell$) be different rational integers and let $c>0$ be also any fixed integer. 
Consider the binary form $F(X,Y)$ of degree $2\ell$ defined by 
$$
F(X,Y)=\prod_{j=1}^\ell (X-c_jY)^2+cY^{2\ell}.
$$
 Here $F(x,y)>0$ for all $(x,y)\in\R^2\setminus\{(0,0)\}$.
Moreover, for $j=1,2,\dots,\ell$, we have $F(c_j,1)=c $,
and the minimum value on the real axis of the associated polynomial $f(X)$, defined by $F(X,1)$, is $c$. 
%%%%%%%%%%%%%%%%%%%%%%%%%%%%%%%%%%%%%%%%%%%%%%%%%

 \section{On cyclotomic polynomials }
 
The cyclotomic polynomials $\phi_n(X)\in\Z[X]$, $n\ge 1$, are defined by the formula
\begin{equation}\label{Cyclotomic1}
 \phi_n(X)=\prod_{\zeta\in E_n} (X-\zeta) 
\end{equation}
where $E_n$ is the set of primitive roots of unity of order $n$. 
 One can also define them via the recurrence provided by 
\begin{equation}\label{Cyclotomic2}
X^n-1=\prod_{d\mid n} \phi_d(X).
\end{equation}
The degree of $\phi_n(X) $ is $\varphi(n)$, where $\varphi$ is Euler's totient function. 
We will always suppose that $n\ge 3$, whereupon $\varphi(n)$ is always even. 
For $n\ge 3$, the polynomial $\phi_n(X)$ has no real 
root. \bigskip

Two very important formulas for cyclotomic polynomials are the following ones: when 
 $n$ is an integer $\geq 1$ written as $n=p^r m$ with $p$ a prime number dividing $n$ and with $m$ such that 
 $\mbox{\sc gcd}(p,m)=1$, 
 we have
\begin{equation}\label{Cyclotomic3}
\phi_{n}(X)
= \frac{\phi_{m}\bigl (X^{p^r}\bigr ) }
 { \phi_{m} \bigl ( X^{p^{r-1} } \bigr) } \qquad \mbox{and} \qquad 
\phi_{n}(X)={\phi_{pm}\bigl(X^{p^{r-1}}\bigr) }.
\end{equation}

\smallskip

For our purposes, we will use the following properties:

(i) The $n$-$th$ cyclotomic polynomial can be defined by
\begin{equation}\label{Equation:Mobius}
\phi_n(X)=
\prod_{d\mid n} (X^d-1)^{\mu(n/d)},
\end{equation}
where $\mu$ is the M\oe bius function.

(ii) Let $n=2^{e_0}p_1^{e_1}\cdots p_r^{e_r}$ where $p_1,\dots,p_r$ are different odd primes, $e_0\ge 0$, $e_i\ge 1$ for $i=1,\dots,r$ and $r\ge 1$. Denote by $R$ the radical of $n$, namely 
$$
 R=\left\{ \begin{array}{rll}
 2p_1\cdots p_r&\mbox{if } e_0\ge 1,\\[2mm]
 p_1\cdots p_r&\mbox{if } e_0=0. \end{array} \right.
$$ 
Then, 
\begin{equation}\label{CyclotomicRadical}
\phi_n(X) = \phi_{R} (X^{n/R}).
\end{equation}

(iii) Let $n=2m$ with $m$ odd $\ge 3$. Then
\begin{equation}\label{CyclotomicPair}
\phi_n(X) = \phi_{m} (-X).
\end{equation}
 %%%%%%%%%%%%%%%%%%%%%%%%%%%%%%%%%%%%%%%%%%%%%%%%%%%

%%%%%%%%%%%%%%%%%%%%%%%%%%%%%%%%%%%%%%%%%%%%%%%

\section{The invariants $c_n$}\label{cn}
 The real number $c_n$, which we define by 
$$
c_n=\inf_{t\in \R} \phi_n(t),
$$
is always $>0$ for $n\geq 3$; this invariant $c_n$ will play a major role in this paper. 
Since the cyclotomic polynomials are reciprocal, we deduce from 
($\ref{equation:reciproque}$)
\begin{equation}\label{Equation:interval-1+1}
c_n=\inf_{-1\le t\le 1} \phi_n(t).
\end{equation}
 
\begin{prop}\label{Prop:cn}
Let $n\ge 3$. Write 
$$
n=2^{e_0}p_1^{e_1}\cdots p_r^{e_r}
$$
 where $p_1,\dots,p_r$ are odd primes with $p_1<\cdots<p_r$, $e_0\ge 0$, $e_i\ge 1$ for $i=1,\dots,r$ and $r\ge 0$.
 \\
 \mbox{\rm (i)} For $r=0$, we have $e_0\ge 2$ and $c_n=c_{2^{e_0}}=1$.
\\
\mbox{\rm (ii)} For $r\ge 1$ we have 
$$
c_n=c_{p_1\cdots p_r}
\ge p_1^{-2^{r-2}}.
$$
\end{prop}

Here are the first values of $c_n$ for $n$ odd and squarefree, with for each $n$ a value of $t_n \in ]\!-1,1[$ such that $c_n=\phi_n(t_n)$:
 \bigskip\bigskip

\begin{center}
\noindent $
\begin{array}{|c||c|c|} \hline 
n& c_n&\; t_n\\[2mm]
\hline\hline
3&\! \! \! 0.75&\! \!\! -0.5\\
5&0.673...&-0.605...\\
7&0.635...&-0.670...\\
11&0.595...&-0.747...\\
13&0.583...&-0.772... \\
15&0.544...&-0.792...\\ 
17&0.567...&-0.808...\\ 
\hline 
\end{array}$\quad\;
$\begin{array}{|c||c|c|}\hline
n&c_n&\; t_n\\[2mm]
\hline\hline
19&0.562...&-0.822...\\ 
21&0.496...&-0.834...\\
23&0.553...&-0.844...\\
29&0.544...&-0.867...\\
31&0.541...&-0.873...\\ 
33&0.447...&-0.879...\\ 
35&0.375...&-0.884...\\ 
\hline
\end{array}
$\quad\;
$
\begin{array}{|c||c|c|}\hline
n&c_n&t_n\\[2mm]
\hline\hline
37&0.536...&-0.889... \\
39&0.786...&-0.954...\\
41&0.533...& -0.897...\\
43&0.531...&-0.900...\\
47&0.529...&-0.907...\\
51&0.778...&-0.964...\\
53&0.526...&-0.915...\\
\hline 
\end{array}$
\end{center}

\bigskip

\centerline{{\bf Tables 3}}

\bigskip 
 
{\sc Proof of Proposition $\ref{Prop:cn}$.} 
 In view of the properties 
 ($\ref{CyclotomicRadical}$) and ($\ref{CyclotomicPair}$), 
we may restrict to the case where $n$ is odd and squarefree. 

We plan to prove
\begin{equation}\label{Equation:Proposition4.2}
 \phi_{p_1p_2\cdots p_r}(t)\ge \frac{1}{p_1^{2^{r-2}}}
\end{equation}
for $r\ge 1$ and $-1\le t\le 1$. 

We start with the case $r=1$. Let $p$ be an odd prime.
For $-1\le t\le 0$, we have 
$1\le 1-t^p\le 1-t\le 2$, hence 
\begin{equation}\label{Equation:tp-10}
\frac{1}{2}\le \phi_{p}(t)\le 1.
\end{equation}
For $0\le t\le 1$, we have 
$0\le 1-t\le 1-t^p\le 1$ and $\phi_p(t)=1+t+t^2+\cdots+t^{p-1}$,
whereupon 
\begin{equation}\label{Equation:tp0-1}
1\le \phi_{p}(t)\le p.
\end{equation}
We deduce $1/2\le \phi_p(t)\le p$ for $-1\le t\le 1$. Since $c_3=3/4$, this completes the proof of ($\ref{Equation:Proposition4.2}$) for $r=1$. 

Assume now $r\ge 2$. Using ($\ref{Equation:Mobius}$) for $n=p_1\cdots p_r$, we express $\phi_n(t)$ as a product of $2^{r-1}$ factors, half of which are of the form $\phi_{p_1}(t^d)$ while the other half are of the form $1/\phi_{p_1}(t^d)$, where $d$ is a divisor of $p_2p_3\cdots p_r$. 

For $t$ the interval $[-1,0]$, using ($\ref{Equation:tp-10}$), we have 
$$
\frac12\le \phi_{p_1}(t)\le 1 
\quad\mbox{and}\quad 
\frac12 \le \phi_{p_1}(t^d)\le 1,
$$
hence 
\begin{equation}\label{Equation:Phip1p2pr1}
\frac{1}{2^{2^{r-2}}}\le \phi_{p_1p_2\cdots p_r}(t)\le 2^{2^{r-2}} .
\end{equation}
For $t$ the interval $[0,1]$, using ($\ref{Equation:tp0-1}$), we have 
$$
1\le \phi_{p_1}(t)\le p_1 
\quad\mbox{and}\quad 
1\le \phi_{p_1}(t^d)\le p_1,
$$
whereupon 
\begin{equation}\label{Equation:Phip1p2pr2}
\frac{1}{p_1^{2^{r-2}}}\le \phi_{p_1p_2\cdots p_r}(t)\le p_1^{2^{r-2}} .
\end{equation}
From ($\ref{Equation:Phip1p2pr1}$) and ($\ref{Equation:Phip1p2pr2}$), we conclude that 
($\ref{Equation:Proposition4.2}$) is true. Thanks to ($\ref{Equation:interval-1+1}$), ($\ref{Equation:Proposition4.2}$) can be written
$$\centerline{$ \displaystyle
\log c_n\ge -2^{r-2}\log p_1.$}\hfill \Box 
$$
 
 We need an auxiliary result.

\begin{lem} \label{Lemma:MinorationPhin}
For any odd squarefree integer $n=p_1\cdots p_r$ with $p_1<p_2<\cdots<p_r$ satisfying $n\geq 11$ and $n\neq 15$, we have
\begin{equation}\label{Equation:MinorationPhin}
\varphi(n) >2^{r+1} \,\log\, p_1.
\end{equation}
\end{lem}
\begin{proof}[\indent \sc Proof.]
If $r=1$, the number $n$ is a prime $\ge 11$ and ($\ref{Equation:MinorationPhin}$) is true with $p_1=n$. 
 If
$r=2$, $n\neq 15$, we have $p_2\geq 7$, hence
$$
\varphi(p_1p_2)=(p_1 -1)(p_2 -1) > 6(p_1 -1) >8\log \, p_1,
$$ 
whereupon ($\ref{Equation:MinorationPhin}$) is true. 

Assume $r\ge 3$. We have 
$$
\varphi(n)=(p_1-1)(p_2-1)\cdots(p_r-1)> (p_1-1) 2^{2(r-1)}\ge(p_1-1)2^{r+1} >2^{r+1}\log p_1.
$$
This completes the proof of Lemma $\ref{Lemma:MinorationPhin}$. 
\end{proof}

We deduce the following consequence.

\begin{prop}\label{Proposition:propcn}
For $n\ge 3$, we have
$$
c_n\ge (\sqrt{3}/2)^{\varphi(n)}.
$$
\end{prop}

This lower bound is best possible, since there is equality for $n=3$ (and $n=6$). 

\begin{proof}[\indent {\sc Proof of Proposition $\ref{Proposition:propcn}$}] 
It suffices to check the inequality when $n$ is an odd squarefree integer, say $n=p_1\cdots p_r$ where $p_1<p_2< \cdots < p_r$ with $r\ge 1$. 
This lower bound is true for $n=3$ (with equality, since $c_3=3/4$), and also for $n=5$, for $n=7$ and for $n=15$, since
$$
c_5>0.6>(\sqrt{3}/2)^4,\qquad
c_7>0.6>(\sqrt{3}/2)^6,\qquad
c_{15}>0.5>(\sqrt{3}/2)^8.
$$ 
Using Proposition $\ref{Prop:cn}$(ii) and Lemma $\ref{Lemma:MinorationPhin}$, we have
$$
8\log c_n \ge -2^{r+1} \log\, p_1 \geq -\varphi(n),
$$
whereupon
$$
c_n \geq e^{- \varphi(n)/8}\geq \left(\displaystyle \frac{\sqrt{3}}{2}\right)^{\varphi(n)}
$$
since $ \log (2/\sqrt{3})>1/8$.
\end{proof}

Proposition $\ref{Proposition:propcn}$ will be sufficient for the proofs of Theorem $\ref{Thm:1}$, Theorem $\ref{Thm:1suite}$ and Lemma \ref{Lemma:varphi>2}. However, it may be of independent interest to state further properties of $c_n$, which are easy to prove.
%, which we prove in \S\ref{Appendice}. 

\bigskip
 
For $p$ an odd prime number, the derivative $\phi'_p(t)$ of the cyclotomic polynomial $\phi_p(t)$ has a unique real root, this root lives in the interval $]\!-1, -\frac12]$ and will be denoted $t_p$.

\noindent
 $\bullet$ For $p=3$, we have $t_3=-\frac12$.

\noindent
 $\bullet$ For $p$ an odd prime number, one has $\, c_p=p t_p^{p-1}$.
 \\
 $\bullet$ The sequence $(t_p)_ {p\; \mbox{\footnotesize odd prime} } $
is decreasing and converges to $-1$; in fact, we have 
$$
%t_p=-1+\frac{\log p}{p}+\frac{\theta_p}{p} \quad\hbox{with}\quad \lim_{p\rightarrow\infty}\theta_p=\log 2.
-1+\frac{\log (2p)}{p}-\frac{(\log(2p))^2}{2p^2} <t_p<-1+\frac{\log (2p)}{p}+\frac{\log(2p)}{p^2}\cdotp
$$
%for sufficiently large $p$.
\\
 $\bullet$ 
 The sequence $(c_p)_ {p\; \mbox{\footnotesize odd prime} } $
 is decreasing and converges to $1/2$; in fact, we have 
$$
c_p=\frac12+\frac{1+\log (2p)}{4p} +\frac{\nu_p(\log p)^2}{p^2} \quad\hbox{with}\quad 
 |\nu_p|\le \frac{1}{4}\cdotp
 %\longrightarrow \frac{1}{8} \quad\hbox{when}\quad p\longrightarrow\infty.
$$ 
 $\bullet$ 
 Let $p_1$ and $p_2$ be two primes. We have
$$
c_{p_1p_2}\ge \frac{1}{p_1}\cdotp
$$
Further, for any prime $p_1$, we have
$$
\lim_{p_2\rightarrow \infty}c_{p_1p_2}= \frac{1}{p_1}\cdotp
$$
$\bullet$ We have $\displaystyle \liminf_{n\rightarrow \infty} c_n=0$ and $\displaystyle \limsup_{n\rightarrow \infty} c_n=1$.
 
% {\bf Je ne vois rien \`a conjecturer pour $t_p$}
% \bigskip
% \noindent

%%%%%%%%%%%%%%%%%%%%%%%%%%%%%%%%%%%%%%%%%%%%%
\section{Proof of Theorems \ref{Thm:1} and \ref{Thm:1suite} }\label{S:ProofMainResults}
 
{\sc $\quad$ Proof of Theorem $\ref{Thm:1}$. }
Assume 
 $$
 \Phi_n(x,y)=m
 $$
with $n\geq 3$ and $\max\{|x],|y|\} \geq 2 $. Using Lemma $\ref{Lemma:PositiveDefiniteBinaryForms}$, 
we deduce 
\begin{equation}\label{Equation:cnmaxxy}
c_n\max\{|x],|y|\}^{\varphi(n)}\le m.
\end{equation} 
From Proposition $\ref{Proposition:propcn}$ we deduce 
\begin{equation}\label{Equation:sqrt3maxxy}
\left(\frac{\sqrt{3}}{2}
\max\{|x|,|y|\}\right)^{\varphi(n)} \le m.
\end{equation} 
Since $\max\{|x],|y|\} \geq 2 $, we deduce the desired upper bound for $\varphi(n)$: 
$$
3^{\varphi(n)/2} \le m.
$$
Using again ($\ref{Equation:sqrt3maxxy}$), we deduce
$$ \centerline{$ \displaystyle
\max\{|x|,|y|\} \le \frac{2}{\sqrt{3}} \,m^{1/\varphi(n)}.$} \hfill \Box 
$$ 
 
 \begin{proof}[\indent {\sc Proof of Theorem $\ref{Thm:1suite}$}]
We first prove that if the triple $(n,x,y)$ satisfies 
$$
n\ge 3, \quad \max\{|x|,|y|\}\ge 2 \quad \mbox{and} \quad \Phi_n(x,y)<7^{\varphi(n)/2},
$$
then $\max\{|x|,|y|\}=2$. Using MAPLE \cite{MAPLE}, we check that this property is verified for $n\in\{3,5,7,15\}$, namely, each of the inequalities 
$$
\Phi_3(x,y)<7,\quad \Phi_5(x,y)<7^2,\quad \Phi_{7}(x,y)<7^3,\quad \Phi_{15}(x,y)<7^4
$$
implies $\max\{|x|,|y|\}=2$. 

For $n$ an odd squarefree integer $\not \in\{3,5,7,15\}$, according to ($\ref{Equation:MinorationPhin}$), we have 
$$
\varphi(n)> 2^{r+1}\log p_1.
$$
Since $\log (3/\sqrt{7})>1/8$, we deduce from ($\ref{Equation:cnmaxxy}$) and Proposition $\ref{Prop:cn}$
that the assumption $\Phi_n(x,y)<7^{\varphi(n)/2}$ implies 
\begin{align}\notag
\varphi(n)\log\max\{|x|,|y|\} &\le \log\Phi_n(x,y)- \log c_n 
\\[3mm]
\notag
&< \frac{\varphi(n)}{2}\log 7 +2^{r-2}\log p_1
\\[3mm]
\notag
&
<
 \left(\frac{1}{2}\log 7 +\frac{1}{8}\right)\varphi(n)<\varphi(n)\log 3,
\end{align}
hence
$\max\{|x|,|y|\}<3$ and therefore $\max\{|x|,|y|\}=2$.
 Since $2\log 2<\log 7$, we deduce that the assumptions $n\ge 3$, $\max\{|x|,|y|\}\ge 2$, and $\Phi_n(x,y)\le 2^{ \varphi(n)}$ imply $\max\{|x|,|y|\}=2$. 

Let $\theta \in ]0,1[$ and let the triple $(n,x,y)$ satisfy $n\ge 3$, $\max\{|x|,|y|\}\ge 2$, and $\Phi_n(x,y)\le 2^{\theta \varphi(n)}$.
Therefore
$$
c_n\le 2^{(\theta-1)\varphi(n)}.
$$
Proposition $\ref{Prop:cn}$ implies
$$
(1-\theta) (\log \, 2) \varphi(n) \le 2^{r-2}\log\, p_1.
$$
It remains to check that the odd squarefree integers $n$ satisfying this condition are bounded. Indeed, if $r=1$, then $n=p_1$ satisfies 
$$
2(\log \, 2)(1-\theta)(p_1-1)\le \log p_1,
$$
hence $p_1$ is bounded. If $r\ge 2$, then the condition 
$$ 
(1-\theta) (\log\, 2)(p_1 -1)(p_2 -1)(p_3 -1) \cdots (p_r -1) \leq 2^{r-2} \log p_1 
$$
shows that $p_1p_2\cdots p_r$ is bounded. 
\end{proof}

The proofs of Theorem $\ref{Thm:2}$ and Corollary $\ref{Corollary}$ will use the following result, the proof of which rests on Proposition $\ref{Proposition:propcn}$. 

\begin{lem}\label{Lemma:varphi>2}
Let $d>2$. There exists an effectively computable positive constant $C(d)$ such that the number of triples $(n,x,y)$ in $\Z^3$ which are satisfying $\varphi(n)\ge d$, $\max\{|x|,|y|\}\ge 2$ and $\Phi_n(x,y)< N$ is bounded by $C(d) N^{2/d}$.
\end{lem}

Given a positive integer $N$ and a binary form $F(X,Y)$ of degree $d$, with integer coefficients and nonero discriminant, denote by $R_F(N)$ 
the number of integers of absolute value at most $N$ which are represented by $F(X,Y)$. In \cite{SY}, the authors quote the foundational work of Fermat, Lagrange, Legendre and Gauss concerning the case where $F$ is a binary quadratic form, and a result of Erd\H{o}s and Mahler (1938) for forms of higher degrees. They prove that for $d\ge 3$, there exists a positive constant $C_F>0$ such that $R_F(N)$
is asymptotic to $C_F N^{2/d}$. 
In Lemma $\ref{Lemma:varphi>2}$, we deal with a sequence of forms having no real zero, 
a situation which is easier to deal with. 

\begin{proof}[\indent \sc Proof of Lemma $\ref{Lemma:varphi>2}$.]
If $m< N$ is represented by $\Phi_n(x,y)$ with $\varphi(n)\ge d$, then we have $\Phi_n(x,y)< N$, hence by ($\ref{Equation:cnmaxxy}$) we have $c_n2^{\varphi(n)}< N$. From Proposition $\ref{Proposition:propcn}$ we deduce $3^{\varphi(n)/2} < N$, whereupon 
$\varphi(n) < (2\log N)/\log 3$. Next, from ($\ref{Equation:sqrt3maxxy}$) we deduce 
$$
\max\{|x|,|y|\} \le \frac{2}{\sqrt{3}}\, m^{1/\varphi(n)} < \frac{2}{\sqrt{3}}\, N^{1/\varphi(n)}\le \frac{2}{\sqrt{3}}\, \,N^{1/d}, 
$$
which proves that for each $n$, the number of $(x,y)$ is bounded by $(16/3)N^{2/d}$. From ($\ref{Equation:1.161}$) we deduce that the number of triples $(n,x,y)$ in $\Z^3$ which satisfy $\varphi(n)\ge d$, $\max\{|x|,|y|\}\ge 2$ and $\Phi_n(x,y)< N$ is bounded by $29 N^{2/d}(\log N)^{1.161}$. 

To complete the proof of Lemma $\ref{Lemma:varphi>2}$, we consider two cases. If there is no $n$ with $\varphi(n)=d$, then we deduce the sharper upper bound $29 N^{2/(d+1)}(\log N)^{1.161}$. If the set $\{n_1,n_2,\dots,n_k\}$ of integers $n$ satisfying $\varphi(n)=d$ is not empty, for $1\le j\le k$ the number of couples $(x,y)$ in $\Z^2$ satisfying $\max\{|x|,|y|\}\ge 2$ and $\Phi_{n_j}(x,y)< N$ 
is bounded by $(16/3)N^{2/d}$, while the number of triples $(n,x,y)$ in $\Z^3$ with $\varphi(n)>d$, $\max\{|x|,|y|\}\ge 2$ and $\Phi_n(x,y)< N$ is bounded by $29 N^{2/(d+1)}(\log N)^{1.161}$. Since $k$ is bounded in terms of $d$, Lemma $\ref{Lemma:varphi>2}$ follows.
\end{proof}

%%%%%%%%%%%%%%%%%%%%
\section{Proof of Theorem \ref{Thm:2} and Corollary \ref{Corollary}}\label{proofthm13}
We start from the easy inequality concerning the cardinality of the union of finite sets. We have 
\begin{multline}\label{mult}
\Bigl\vert \, \bigl\vert \,\mathcal A( \Phi_{\{n\geq 3\}}; N)\, \bigr\vert - \Bigl( \bigl\vert \,\mathcal A( \Phi_{3}; N)\, \bigr\vert +
\bigl\vert \,\mathcal A( \Phi_{4}; N)\, \bigr\vert -\bigl\vert \,\mathcal A( \Phi_{3}; N)\cap \mathcal A( \Phi_{4}; N) \, \bigr\vert \Bigr) \,\Bigl\vert 
\\
\leq \bigl\vert \, \bigcup_{\varphi (n) \geq 4}\mathcal A( \Phi_{n}; N)\, \bigr\vert .
\end{multline}
By Lemma \ref{Lemma:varphi>2} the right--hand side of \eqref{mult} is $O (N^\frac 12)$ which is absorbed by the error term of the formula
\eqref{infexpan}. So we are led to study the cardinalities of three sets $\mathcal A( \Phi_{ 3}; N)$, $ \mathcal A( \Phi_{4}; N)$ and $ \mathcal A( \Phi_{3}; N)\cap \mathcal A( \Phi_{4}; N)$. For algebraic considerations, it is better to consider for $k\in \{3,\, 4\}$ the larger sets 
$$
\tilde{ \mathcal A} (\Phi_k; N) := \bigl\{ m\in \mathbb N\mid m\leq N, m= \Phi_n (x,y) \text{ for some } (x,y)\in \mathbb Z^2 \backslash
\{(0,0)\} 
\bigr\},
$$
which differ from 
$
\mathcal A (\Phi_k; N)$ by at most two terms. In conclusion, the proof of Theorem \ref{Thm:2} will be complete (with $\alpha_h=\alpha_h^{(3)}+\alpha_h^{(4)}$, $h\geq 0$)
%, while the $\beta_h$ of Theorem \ref{Thm:2} are the $\beta_h$ of Proposition \ref{964}
as soon as we prove 
\begin{prop}\label{964}
There exist three sequences of real numbers $(\alpha_h^{(3)})$, $(\alpha_h^{(4)} )$ and $(\beta_h)$ ($h\geq 0$) with $\alpha_0 ^{(3)}, \alpha_0^{(4)}$ and $\beta_0 >0$, such that for every 
 for $M \geq 0$, the following equalities holds uniformly for $N \geq 2$
\begin{multline}\label{infexpan3ou4}
\bigl\vert \,\tilde{\mathcal A}( \Phi_{k}; N)\, \bigr\vert = \frac{N}{(\log N)^\frac{1}{2}} 
 \left\{ \, \alpha_0^{(k)} + 
 \frac{\alpha_1^{(k)}}{(\log N)} + \cdots \right.\\ \left. + 
 \frac{\alpha_M^{(k)}}{(\log N)^M}\ + O \left( \frac{1}{(\log N)^{M+1}}\right)
 \right\} \ (k=3, \; %\text{ or } 
 4)
\end{multline}
and 
\begin{multline}\label{infexpan3et4}
\bigl\vert \,\tilde{\mathcal A}( \Phi_{ 3}; N)\,\cap \, \tilde{\mathcal A}( \Phi_{ 4}; N)\, \bigr\vert = \frac{N}{(\log N)^\frac{3}{4}} 
 \left\{ \beta_0 + 
 \frac{\beta_1}{\log N}\, + \cdots \right.\\ \left.+\frac{\beta_M}{(\log N)^M}
 + O \left( \frac{1}{(\log N)^{M+1}}\right)
 \right\}.
\end{multline}
\end{prop}
The proof of this proposition will be achieved in the next three subsections. We will exploit the fact that $\Phi_3$ and $\Phi_4$ are binary quadratic forms, which also are the norms of integers of imaginary quadratic fields with class number one. Finally the characteristic functions of the sets $\tilde{\mathcal A}( \Phi_{k}; \infty)$ for $k\in \{ 3, \, 4 \}$ are studied by analytic methods via the theory of Dirichlet series.

\subsection{Algebraic backgrounds}
We fix some notations. 
 The letter $p$ is reserved for primes. If $a$ and $q$ are two integers, we denote by 
 %$\mathcal P_{a,q}$ the set of primes $p \equiv a\bmod q$ and by 
 $N_{a,q}$ any integer $\geq 1$ satisfying the condition
$$
p \mid N_{a,q} \Longrightarrow p \equiv a\bmod q.
$$
\begin{prop}\label{prop1} The following equivalences hold true.
\begin{enumerate} 
\item[{\rm (i)}]
An integer $n\geq 1$ is of the form
$$
n=\Phi_4(x,y)=x^2+y^2
$$
if and only if there exist integers $a\geq 0$, $N_{3,4}$ and $N_{1,4}$ such that
$$
n= 2^a \,N_{3,4}^2\, N_{1,4}.
$$
 \item[{\rm (ii)}]
An integer $n\geq 1 $ is of the form
$$
n=\Phi_3 (u,v) = \Phi_6 (u,-v)=u^2+uv +v^2
$$
if and only if there exist integers $b\geq 0$, $N_{2,3}$ and $N_{1,3}$ such that
$$
n= 3^b \,N_{2,3}^2\, N_{1,3}.
$$
\item[{\rm (iii)}]
An integer $n\geq 1$ is simultaneously of the forms
$$
 n=\Phi_3 (u,v) =u^2+uv +v^2 \text{ and } n=\Phi_4 (x,y) =x^2+y^2 
$$
if and only if there exist integers $a, \, b\geq 0$, $N_{5,12}$, $N_{7,12}$, $N_{11,12}$ and $N_{1,12}$ such that
$$
n= \Bigl(2^a\, 3^b\, N_{5,12}\,N_{7,12}\, N_{11,12}\Bigr)^2 N_{1,12}.
$$
\end{enumerate}
\end{prop}
 Proposition \ref{prop1}(i) is famous (see \cite[Theorem 366]{H--W} for instance).
It can be proved by detecting primes in the ring of the Gaussian integers $\mathbb Z [i]$ of the quadratic field $\mathbb Q (i)$. This ring is principal and 
 the norm of the element $x+iy$ is given by the quadratic form $\Phi_4(x,y) =x^2+y^2$. The quadratic field $\mathbb Q (\sqrt{-3})$ has similar properties: its associated ring of integers is a principal domain equal to $\mathbb Z [j]$ with $j=(-1+ \sqrt {-3})/2$. The primes of $\mathbb Z [j]$ 
 (also called Eisenstein primes) are detected by the values of the Kronecker symbol $(-3/p)$ and the norm of the element $u+vj$ of $\mathbb Z [j]$ is equal to $\Phi_3 (u, -v) = \Phi_6 (u,v) =u^2+uv +v^2$. This gives Proposition \ref{prop1}(ii). For instance this statement is a particular case of \cite[Th\' eor\`eme 3, p. 267]{BoCha}
and it is implicitely contained in \cite [Theorem 254]{H--W}, 
\cite[Exercise 2, p. 308]{Hu}.

Combining Proposition \ref{prop1}(i)and \ref{prop1}(ii), we deduce Proposition \ref{prop1}(iii) directly. \qed
%%%%%%%%%%%%%%%%%%%
\subsection{Analytic background}
Our main tool is based on the Selberg--Delange method. The following version is a weakened form of the quite general result due to Tenenbaum (see \cite[Theorem 3, p. 185]{Ten}).
%\footnote{Check this reference in the english book}
 It gives an asymptotic expansion of the summatory function of a sequence $(a_n)$ when the attached Dirichlet series can be approached by some power of the $\zeta$-function
in a domain slightly larger than the half--plane $\{s\in \mathbb C\mid \Re s \geq 1\}$. We have 
\begin{prop}\label{bookTen} Let $s=\sigma +it$ be the complex variable and let 
$$
F(s):= \sum_{n\geq 1} a_n n^{-s}
$$be a Dirichlet series such that

\noindent $ \bullet$ the coefficients $a_n$ are real nonnegative numbers,

\noindent $ \bullet$ there exist $z \in \mathbb C$, $c_0 >0$, $\delta >0$ and $K>0$, such that the function 
$$
G(s):= F(s) \zeta (s)^{-z}
$$
 has a holomorphic continuation in the domain $\mathcal D$ of the complex plane, defined by the inequality
\begin{equation}\label{domain}\sigma > 1 -\frac{c_0}{ 1 +\log (1+ \vert t \vert)}\, ,
\end{equation} and satisfies the inequality
\begin{equation}\label{ineq}
\vert G(s)\vert \leq K (1+ \vert t \vert)^{1-\delta}
\end{equation}
for every $s \in \mathcal D$.

Then there exists a sequence of real numbers $(\lambda_k)$ ($k\geq 0$) such that
for all $M\geq 1$, uniformly for $x \geq 2$, we have the equality
$$
\sum_{1\leq n \leq x} a_n =x (\log x)^{z-1} \left\{ \sum_{0\leq k \leq M}\frac{\lambda_k}{(\log x)^k} + O \Bigl( \frac{1}{(\log x)^{M+1}}\Bigr)\,\right\}\cdot 
$$
In particular, we have the equality 
$$
\lambda_0= \frac{1}{\Gamma( z)} G(1).
$$
\end{prop}
\subsection{Proof of Proposition \ref{964}} We restrict ourselves to the proof of \eqref{infexpan3et4} since the proof of \eqref{infexpan3ou4} is simpler. 
Let $\xi_n$ be the characteristic function of the set of integers $n\geq 1$ which are simultaneously represented by $\Phi_3$ and $\Phi_4$.
Let $F(s)= \sum_n \xi_n n^{-s}$ be the associated Dirichlet series. Note the equality
$$ \bigl\vert \,\tilde{\mathcal A}( \Phi_{ 3}; N)\,\cap \, \tilde{\mathcal A}( \Phi_{ 4}; N)\, \bigr\vert 
= \sum_{n\leq N} \xi_n.
$$
 By the third part of Proposition \ref{prop1}, $F(s)$ factorizes in the product
\begin{equation}\label{1064}
F(s)= H(s) \Pi (s) 
\end{equation}
 with 
 \begin{align}\label{defH}
 H(s)&= \left(1 -\frac{1}{4^s}\right)^{-1}\left(1 -\frac{1}{9^s}\right)^{-1}
\prod_{p\, \equiv \, 5,\, 7,\, 11 \bmod 12} \left(1 -\frac{1}{p^{2s}}\right)^{-1},\\
\Pi (s)& = \prod_{p\, \equiv \, 1 \bmod 12} \left(1 -\frac{1}{p^s}\right)^{-1}.\label{defPi}
 \end{align} 
 %%%%%%%%%%%%%%%%%%%%%

 The function $H$ is holomorphic for $\sigma >1/2$ and uniformly bounded for $\sigma \geq 3/4$. The infinite product $\Pi (s)$ is absolutely convergent for $\sigma >1$ and we want to study the behavior of this product in the vicinity of the singularity $s=1$. 
 To detect among the primes $p\ge 5$ those which are congruent either to $1$ modulo $12$
or to $5,7,11$ modulo 12, we use the formula 
 \begin{equation}\label{?}
 \frac{1}{4} \left( \ 1 + \left(\frac{-3}{p}\right) + \left( \frac{-4}{p} \right)+ \left(\frac{12}{p}\right) \, \right)=
 \begin{cases} 1 & \text { if } p \equiv 1 \bmod 12,\\
 0 & \text{ if } p\equiv 5, \, 7 ,\, 11 \bmod 12.
 \end{cases}
 \end{equation} 
 Inserting \eqref{?} into \eqref{defPi}, we deduce that for $\sigma >1$ we have the equality
 \begin{multline*}
 \Pi (s) = \prod_{p \geq 5} \left\{ \left( 1-\frac{1}{p^s}\right) 
 \left( 1-\frac{(-3/p)}{p^s}\right) \left( 1-\frac{(-4/p)}{p^s}\right) \left( 1-\frac{(12/p)}{p^s}\right)\right\}^{- \frac{1}{4}}\\
 \times \prod_{p\, \equiv \, 5,\, 7,\, 11 \bmod 12} 
 \left( 1 -\frac{1}{p^{2s}}\right)^\frac{1}{2}\cdot
 \end{multline*}
 Completing the first infinite product with the factors associated with the primes $p=2$ and $p=3$ to obtain the $\zeta$--function and some $L$--functions, we deduce that for $\sigma >1$, $\Pi (s)$ satisfies the equality
 \begin{equation}\label{equaPi}
 \Pi (s) = H_1 (s) \, \zeta(s)^\frac{1}{4} \,L(s, (-3/\cdot))^\frac{1}{4}\, L(s, (-4/\cdot))^\frac{1}{4}\, L(s, (12/\cdot))^\frac{1}{4},
 \end{equation}
 with
 $$
 H_1 (s) = \left(1 -\frac{1}{4^s}\right)^\frac{1}{4} \, \left(1-\frac{1}{9^s}\right)^\frac{1}{4}\, \prod_{p\, \equiv \, 5,\, 7,\, 11 \bmod 12} 
 \left( 1 -\frac{1}{p^{2s}}\right)^\frac{1}{2}.
 $$
 By \eqref{1064}, \eqref{defH}, \eqref{defPi} and \eqref{equaPi}, we deduce that $F(s)$ satisfies for $\sigma >1$ the equality
 \begin{equation}\label{equaF}
 F(s) = H_2 (s) \zeta(s)^\frac{1}{4} \,L(s, (-3/\cdot))^\frac{1}{4}\, L(s, (-4/\cdot))^\frac{1}{4}\, L(s, (12/\cdot))^\frac{1}{4},
 \end{equation}
 with 
 $$
 H_2 (s) = \Bigl(1 -\frac{1}{4^s}\Bigr)^{-\frac{3}{4}} \, \Bigr(1-\frac{1}{9^s}\Bigr)^{-\frac{3}{4}}\, \prod_{p\equiv 5,\, 7,\, 11 \bmod 12} \Bigl( 1 -\frac{1}{p^{2s}}\Bigr)^{-\frac{1}{2}}.
$$
The function $H_2$ is holomorphic for $\sigma >1/2$ and uniformly bounded for $\sigma \geq 3/4$.

 By the classical zero--free region of the Dirichlet $L$--functions, there exists $c_0 >0$ such that in
 the domain $\mathcal D$ defined in \eqref{domain}, the function $$L(s, (-3/\cdot))\, L(s, (-4/\cdot))\, L(s, (12/\cdot))$$ does not vanish. This implies that the function $$G(s):= F(s) \zeta (s)^{-\frac{1}{4}} =H_2 (s)\,L(s, (-3/\cdot))^\frac{1}{4}\, L(s, (-4/\cdot))^\frac{1}{4}\, L(s, (12/\cdot))^\frac{1}{4}
$$
can be extended to a holomorphic function on $\mathcal D$, satisfying the inequality \eqref{ineq}, with $\delta =1/2$, as a consequence of the functional equation and the Phragmen--Lindel\"of convexity principle (see
\cite[Exercise 3, p. 100]{IK} for instance). 

All the conditions of Proposition \ref{bookTen} are satisfied with $z =1/4$ and we obtain \eqref{infexpan3et4} with
$$
\beta_0= H_2 (1) \,L(1, (-3/\cdot))^{\frac{1}{4}}\, L(1, (-4/\cdot))^{\frac{1}{4}}\, L(1, (12/\cdot))^{\frac{1}{4}}/ \Gamma (1/4),
$$
which can be written as 
\begin{multline*}
\beta_0 = 
\Bigl( \frac{3}{2} \Bigr)^\frac{3}{4} \cdot \frac{1}{\Gamma(1/4)}\\
\times 
 \,L(1, (-3/\cdot))^{\frac{1}{4}}\; L(1, (-4/\cdot))^{\frac{1}{4}}\; L(1, (12/\cdot))^{\frac{1}{4}}
 \prod_{p\, \equiv \, 5,\, 7,\, 11 \bmod 12} \left(1 -\frac{1}{p^{2}}\right)^{-\frac 12} \cdot 
\end{multline*}	
Since [OEIS A101455, A073010, A196530]
$$
L(1, (-4/\cdot))= \frac{\pi}{4},
\quad
L(1, (-3/\cdot))= \frac{\pi}{3^{\frac 32}}
\quad\hbox{and}\quad
L(1, (12/\cdot))= \frac{\log (2+\sqrt{3})}{\sqrt{3}},
$$
we deduce
$$
\beta_0=
\frac {3^{\frac 14}}{2^{\frac 54}}\cdot \pi^{\frac 12} \cdot (\log (2+\sqrt{3}) )^{\frac 14} \cdot\frac{1}{\Gamma(1/4)}\cdot
\prod_{p\, \equiv\, 5,\, 7,\, 11 \bmod 12} \Bigl(1 -\frac{1}{p^{2}}\Bigr)^{-\frac 12} .
$$

The proof of \eqref{infexpan3ou4} for $k=3$ and $k =4$ is simpler since the formula to detect the congruences $p\equiv 1 \mod 3$ and $p\equiv 1\bmod 4$
contains only two terms instead of four as in \eqref{?}. In both cases $k=3$ and $k=4$, the parameter $z$ has the value $z=1/2$. This gives (6.2) with 
$$
\alpha_0^{(3)}=
\frac{1}{2^{\frac 12}3^{\frac 14}}\cdot \prod_{p\, \equiv\, 2 \bmod 3} \left(1 -\frac{1}{p^{2}}\right)^{-\frac 12}
$$
and
$$
\alpha_0^{(4)}=
 \frac{1}{2^{\frac 12}}\cdot \prod_{p\, \equiv \, 3 \bmod 4} \left(1 -\frac{1}{p^{2}}\right)^{-\frac 12}. 
 $$ 
Finally, \eqref{infexpan3ou4} is a detailed version of Landau's formula which states that for $N$ tending to infinity, we have 
$$
\bigl\vert \,\tilde{\mathcal A}( \Phi_{4}; N)\, \bigr\vert \sim C \frac{N}{\sqrt{\log N}},
$$
 where $C=\alpha_0^{(4)}=0.7 6 4\, 2 2 3 \,6 5 3 \, 5 8 9 \,2 20\dots $ is the Landau--Ramanujan constant (cf. 
\cite [pp 257-263]{LeVeque} and 
%\cite{OEISA000404,OEISA064533}
%\cite{OEIS}
[OEIS A000404,
%\emph{Numbers that are the sum of 2 nonzero squares}
%\url{https://oeis.org/A000404} and 
OEIS A064533]%
%\emph{Decimal expansion of Landau--Ramanujan constant}
%\url{https://oeis.org/A064533}
). 
Using Pari GP \cite{PariGP}, one checks that the first decimal digits of $\alpha_0^{(3)}$ are $ 0. 638\, 909$, while the first decimal digits of $\beta_0$ are $ 0. 302\, 316$.

\subsection{
 Proof of Corollary $\ref{Corollary}$}
For $N\ge 1$, $a_1+\cdots+a_N$ counts the number of triples $(n,x,y)$ with $n\ge 3$, $\max\{|x|,|y|\}\ge 2$ and $\Phi_n(x,y)\le N$. 
The number of these triples $(n,x,y)$ with $n=4$ is asymptotically $\pi N$.
The number of these triples with $n=3$ is asymptotically $(\pi/\sqrt{3}) N$, and it is the same for $n=6$.
The number of these triples with $\varphi(n)>2$ is $o(N)$, as shown by Lemma $\ref{Lemma:varphi>2}$. Hence 
$$
 a_1+\cdots+a_N\sim \left(1+\frac{2}{\sqrt{3}}\right)\pi N
$$ 
and Corollary $\ref{Corollary}$ with 
$$
\kappa_1=\frac{\pi}{\alpha_0}\left(1+\frac{2}{\sqrt{3}}\right)
$$
 follows from Theorem $\ref{Thm:2}$.
\qed

%%%%%%%%%%%%%%%%%%%%%%%%%%%%%%%%%%%%%%%%%%%%%
\section{Numerical computations}\label{S:NumericalComputations}

From the inequalities in ($\ref{Equation:sqrt3maxxy}$), we deduce that the assumptions $n\ge 3$, $\Phi_n(x,y)\le 20$ and $\max\{|x|,|y|\}\ge 2$ imply 
$$
\left(\frac{\sqrt{3}}{2}
\max\{|x|,|y|\}\right)^{\varphi(n)}\le 20.
$$
We deduce firstly $3^{\varphi(n)/2}\le 20$, hence $\varphi(n)\le 4$, and secondly 
$$
\max\{|x|,|y|\}\le 2\sqrt{20/3},
$$
hence $\max\{|x|,|y|\}\le 5$. It is now again a simple matter of computation with MAPLE \cite{MAPLE} to complete the rest of Table $1$. 
For instance, one can find in Table 4 the values of $(x,y)$ which are the only ones satisfying the stronger condition $\Phi_n(x,y)\le 10$.

\medskip

\begin{small}
\begin{align}
\notag
&
m=3: \; n=3, \; (x,y)=(1,-2), \; (-1,2),\; (2,-1),\; (-2,1), 
\\
\notag
&
m=3:\; n=6,\; (x,y)=(1,2),\; (-1,-2),\; (2,1),\; (-2,-1);
\\
\notag
&
m=4: \; n=3,\; (x,y)=(0,2), \; (0,-2), \; (2,0), \; (2,-2), \; (-2,0), \;(-2,2), \\
\notag
&
m=4: \; n=4,\; (x,y)=(0,2),\; (0,-2),\; (2,0),\; (-2,0), 
\\
\notag
&
m=4:, \; n=6, \;(x,y)=(0,2), \;(0,-2), \;(2,0), \;(2,2), \;(-2,0), \;(-2,-2);
\\
\notag
& 
m=5: \; n=4, \;
(x,y)=(1,2), \;(1,-2), \;(-1,2), \;(-1,-2),\; (2,1), \;(2,-1), \ 
\\ \notag &\qquad\qquad \qquad\qquad \qquad \;\;\; (-2,1), \;(-2,-1);\\
\notag &
m=7: \;n=3, \;
(x,y)=(1,2), \;(1,-3), \;(-1,3), \;(-1,-2), \;(-3,1), \;(3,-1), 
\\
\notag
& 
\qquad\qquad \qquad\qquad \qquad\;
\;(2,1), \;(2,-3), \;(-2,3), \;(-2,-1), \;
(3,-2), \;(-3,2),
\\
\notag
& 
m=7: \;n=6, \;
(x,y)=(1,3), \; (1,-2), \;(-1,2), \; (-1,-3), \;(3,1), \;(-3,-1), \;
\\
\notag
& 
\qquad\qquad \qquad\qquad \qquad\;
\;(2,1), \;(2,-1), \; (2,3), \;(-2,-3), \;(3,2), \;(-3,-2);
\\
\notag
& 
m=8: \; n=4, \;
(x,y)=(2,2), \;(2,-2), \;(-2,2), \;(-2,-2);
\\
\notag
& 
m=9: \;n=3
, \;(x,y)=(0,3), \;(0,-3),\;(3,0), \;(3,3), \;(-3,0),\;(-3,3),
\\
\notag
& 
m=9: \;n=4
, \;(x,y)=(0,3), \;(0,-3),\;(3,0), \;(-3,0),
\\
\notag
& 
m=9: \;n=6
, \;(x,y)=(0,3), \;(0,-3),\;(3,0), \;(3,3), \;(-3,0),\;(-3,3);
\\
\notag
& 
m=10: n=4, \,
(x,y)=(1,3), \;(1,-3), \;(-1,3), \;(-1,-3), \;(3,1), \;(3,-1), \\
\notag 
& \qquad\qquad \qquad\qquad \qquad\;\;\, (-3,1), \;(-3,-1).
\end{align}
\begin{center}
{\bf Table 4} 
\end{center} \end{small} 
%%%%%%%%%%%%%%%%%%%%% Fin de l'appendice

 With similar calculations, we obtain Table $2$. The triples $(n,x,y)$ which contribute to Table $2$ satisfy $\varphi(n)\in\{4,6\}$ and $\max\{|x|,|y|\}\in\{2,3\}$.
%%%%%%%%%%%%%%%%%%%%%%%%%%%%%%%%%%%%%%%%%%%%% 

 Notice that given $h\ge 3$, the smallest value $m_h$ of $m$ for which there exists $(n,x,y)$ with $n\ge 2$, $\max\{|x|,|y|\}\ge h$ and $\Phi_n(x,y)=m$ is 
$$
m_h=\left\{ \begin{array}{llllll}
\displaystyle
\Phi_3\left(\frac{h-1}{2},-h\right)&=&\displaystyle\Phi_3\left(\frac{h+1}{2},-h\right)&=&
\displaystyle
\frac{3h^2+1}{4}&\mbox{\rm if}\; 2\nmid h,\\[5mm]
\displaystyle
\Phi_3\left(\frac{h}{2},-h\right)&=&\displaystyle \frac {3h^2 }{4}&&&\mbox{\rm if}\; 2\mid h.\\
\end{array} \right.
$$

%%%%%%%%% 

 \bigskip 

 {{\bf Acknowledgements.}}
This work was initiated in Lecce in June 2016, and was pursued at the University of the Philippines at Dilliman during a SEAMS school; the second and third authors are grateful to Fidel Nemenzo for the stimulating environment. 
Last, but not least, many thanks to K\'alm\'an Gy\H{o}ry for drawing our attention to his paper \cite{Gyory}
and for his valuable remarks. The second author was supported by an NSERC grant.


\begin{thebibliography}{0EIS}

\normalsize
\baselineskip=17pt

\providecommand{\bysame}{\leavevmode ---\ }
\providecommand{\og}{``}
\providecommand{\fg}{''}
\providecommand{\smfandname}{\&}
\providecommand{\smfedsname}{\'eds.}
\providecommand{\smfedname}{\'ed.}

\bibitem[B--Ch]{BoCha}
Z.I. Borevitch et I.R. Chafarevitch, 
\newblock{\em Th\' eorie des nombres, (French) Traduit par Myriam et Jean-Luc Verley.
 Traduction faite d'apr\`es l'\' edition originale russe},
\newblock Monographies Internationales de Math\' ematiques Modernes, No. {\bf 8}, Gauthier-Villars, Paris, (1967).
 
\bibitem[G]{Gyory}
K.~Gy\H ory,
\emph{Repr\'esentation des nombres entiers par des formes binaires}, 
Publ. Math. Debrecen {\bf 24 (3--4)}, 363--375, (1977).

\bibitem[H--W]{H--W}
 G.H. Hardy and E.M. Wright, 
 \newblock{\em 
An Introduction to the Theory of Numbers, 
Fifth edition},
\newblock The Clarendon Press, Oxford University Press, New York, (1979). 


 \bibitem[H]{Hu}
L.K. Hua, 
\newblock{ \em Introduction to Number Theory, 
Translated from the Chinese by Peter Shiu},
\newblock Springer--Verlag, Berlin-New York, (1982). 

 \bibitem[I--K]{IK}
H.~Iwaniec and E.~ Kowalski,
\newblock{\em Analytic Number Theory,}
\newblock American Mathematical Society Colloquium Publications {\bf 53}, American Mathematical Society, Providence, RI, (2004).

 \bibitem[L]{LeVeque}
W.J.~LeVeque.
\emph{Topics in Number Theory}, Vol. {\bf 2}, Dover, (1956, 2002).
 

 
 \bibitem[M]{MAPLE}
\emph{Maple software}, University of Waterloo, Waterloo, Ontario, Canada.

 \bibitem[M--W]{MW}
M.~Mignotte and M.~Waldschmidt,
\emph{Linear forms in two logarithms and {S}chneider's method, {III},}
Ann. Fac. Sci. Toulouse Math. ({\bf 5}), (suppl.): 43--75, (1989).

\bibitem[N1]{Nagell1}
T.~Nagell, 
\emph{Contributions \`a la th\'eorie des corps et des polyn\^omes cyclotomiques}, Arkiv f\"or Mat. {\bf 5}, (1), (1963), 153--192.

\bibitem[N2]{Nagell2}
T.~Nagell, 
\emph{Sur les repr\'esentations de l'unit\'e par les formes binaires biquadratiques du premier rang}, Arkiv f\"or Mat. {\bf 5}, (6), (1965), 477--521.
 
 \bibitem[OEIS]{OEIS}
N.J.~Sloane, 
\emph{The On--line Encyclopedia of Integer Sequences},
\hfill\newline 
{\tt
\url{https://oeis.org/ }
}


 
 \bibitem[P]{PariGP}
\emph{Pari GP software}, Université Bordeaux I, France\\
\url{https://pari.math.u-bordeaux.fr/}

\bibitem[S--Y]{SY}
C.L.~Stewart and %Stanley Yao Xiao.
S.~Yao Xiao,
\emph{On the representation of integers by binary forms},
 \\
\url{http://arxiv.org/abs/1605.03427} 

\bibitem[T]{Ten}
G. Tenenbaum,
\newblock {\em Introduction to Analytic and Probabilistic Number Theory, (English summary)},
\newblock Translated from the second French edition (1995) by C. B. Thomas, Cambridge Studies in Advanced Mathematics {\bf 46}, Cambridge University Press, Cambridge, (1995). 




 
\end{thebibliography}
\end{document}